\documentclass[a4paper,10pt,reqno]{amsart}
\usepackage{amsmath}
\usepackage{cases}
\usepackage{amsfonts}
\usepackage[colorlinks,linkcolor=blue,citecolor=blue]{hyperref}
\usepackage{latexsym, amssymb, amsmath, amsthm, bbm}
\usepackage[all]{xy}
\usepackage{pgfplots}
\usepackage{enumerate}
\usepackage{mathrsfs}

\DeclareSymbolFont{EulerExtension}{U}{euex}{m}{n}
\DeclareMathSymbol{\euintop}{\mathop} {EulerExtension}{"52}
\DeclareMathSymbol{\euointop}{\mathop} {EulerExtension}{"48}

\allowdisplaybreaks[4]

\def \id{\operatorname{id}}

\def \k{\Bbbk}

\def \dim{\operatorname{dim}}

\def \span{\operatorname{span}}

\usepackage{mathtools}

\numberwithin{equation}{section}

\newtheorem{theorem}{Theorem}[section]
\newtheorem{lemma}[theorem]{Lemma}
\newtheorem{proposition}[theorem]{Proposition}
\newtheorem{corollary}[theorem]{Corollary}
\newtheorem{definition}[theorem]{Definition}

\newtheorem{remark}[theorem]{Remark}

\begin{document}
\title[Derived discrete Hopf algebras]{Derived discrete Hopf algebras with the Chevalley property}

\author[J. Yu]{Jing Yu}
\author[G. Liu]{Gongxiang Liu}
\address{School of Mathematics, Nanjing University, Nanjing 210093, China}
\email{dg21210018@smail.nju.edu.cn}
\address{School of Mathematics, Nanjing University, Nanjing 210093, China}
\email{gxliu@nju.edu.cn}

\thanks{The second author was supported by National Natural Science Foundation of China (NSFC) Grant 12271243.}
\thanks{The first author is the corresponding author.}
\maketitle

\date{}
\begin{abstract}
We try to classify Hopf algebras with the Chevalley property according to their derived representation type. We show that a finite-dimensional indecomposable non-semisimple Hopf algebra $H$ with the Chevalley property is derived discrete if and only if it is isomorphic to $(A(n, 2, \mu, -1))^*$. Besides, we give a description for the indecomposable objects in $\mathcal{D}^b((A(n, 2, \mu, -1))^*)$ and determine their tensor products.
\end{abstract}
\maketitle
\section{Introduction}
The bounded derived categories of algebras have been studied widely since Happel's work \cite{Hap88}. The derived representation type of algebras has received considerable attention. See, for example, \cite{BDF09, BGV21, BM03, BPP17, Bro18, Bru01, CZ19, Zha16}. In particular, the notions of derived discreteness, derived tameness and derived wildness were introduced in \cite{Vos01, del98, GK02}. Meanwhile the tame-wild dichotomy for bounded derived categories of finite-dimensional algebras was established in \cite{BD03}, see also, \cite{Bau07}.

In recent years, much effort was put in the study of monoidal triangulated categories from a Hopf algebraic perspective. See, for instance, \cite{Qi23, Vas24, XL23, ZZ22}. It should be pointed out that the bounded derived categories of finite-dimensional Hopf algebras are monoidal triangulated categories. One of the most important topics in the study of bounded derived categories of Hopf algebras is the decomposition of a tensor product of two indecomposable objects into a direct sum of indecomposables. In compare with the module category, the tensor product in the bounded derived category is more complicated.

Note that a finite-dimensional Hopf algebra is either derived discrete or derived wild (see Lemma \ref{lem:discreteortame}). The bounded derived category of a derived discrete Hopf algebra is easiest to understand. We cannot help but hope to give an accurate description for the tensor product in the bounded derived categories of derived discrete Hopf algebras.

The aim of this paper is to classify derived discrete Hopf algebras with the Chevalley property. Here by the Chevalley property we mean that the radical is a Hopf ideal. We know that the Hopf algebras with the (dual) Chevalley property is a kind of natural generalization of elementary (pointed) Hopf algebras. These Hopf algebras have been studied by many authors. See, for example, \cite{AEG01, AGM17, Li22, YLL24}.

Our main result is the following one:
\begin{theorem}\label{thm:indecom}
Let $H$ be a finite-dimensional indecomposable non-semisimple Hopf algebra over $\k$ with the Chevalley property. Then $H$ is derived discrete if and only if $H\cong (A(n, 2, \mu, -1))^*$ as Hopf algebras.
\end{theorem}

At the same time, gentle algebras form a class of biserial algebras that were introduced in \cite{AS87}. It has transpired since that they naturally appear in many different contexts. See, for example \cite{BC21,Bro12, WS23}. It was shown that for a gentle algebra its bounded derived category is derived tame \cite{BM03}.
In fact, a derived discrete indecomposable non-semisimple Hopf algebra $H$ with the Chevalley property is a gentle algebra. Using \cite[Theorem 3]{BM03}, we describe the indecomposable objects in $\mathcal{D}^b(H)$ and determine their tensor products (see Propositions \ref{prop:indecom} and \ref{prop:tensor}). As a byproduct, we discuss when a gentle quiver admits a Hopf algebra structure (Corollary \ref{coro:gentle}).

The organization of this paper is as follows: Some definitions, notations and results related to derived representation type, monoidal triangulated categories, Hopf algebras with the (dual) Chevalley property and comonomial Hopf algebras are presented in Section \ref{section2}. We devote Section \ref{section3} to classify Hopf algebras with the dual Chevalley property according to their derived representation type. We show that a finite-dimensional indecomposable non-semisimple Hopf algebra $H$ with the Chevalley property is derived discrete if and only if it is isomorphic to $(A(n, 2, \mu, -1))^*$. At last, we give a description for the indecomposable objects in $\mathcal{D}^b((A(n, 2, \mu, -1))^*)$ and determine their tensor products.
\section{Preliminaries}\label{section2}
Throughout this paper $\k$ denotes an \textit{algebraically closed field of characteristic $0$} and all spaces are over $\k$. The tensor product over $\k$ is denoted simply by $\otimes$.
About general background knowledge, the reader is referred to \cite{Hap88} for derived categories, \cite{Mon93} for Hopf algebras, \cite{EGNO15} for monoidal categories and \cite{ASS06} for representation theory.
\subsection{Derived representation type}\label{subsection2.1}
Let $A$ be a finite-dimensional $\k$-algebra. We denote by $A$-mod the category of finitely generated left $A$-modules, by $\mathcal{D}(A)$ its derived category, and by $\mathcal{D}^b(A)$ the derived category of bounded complexes whose terms are in $A$-mod.

With $X\in \mathcal{D}^b(A)$ we associate the sequence $\underline{\dim}X:=(\underline{\dim}H^i(X))_{i\in \Bbb{Z}}$, where $\underline{\dim}H^i(X)$ is the dimension vector of the cohomology module $H^i(X)$, namely, its class in the Grothendieck group $\mathcal{K}_0(A)$ of $A$-mod.

Recall that a \textit{rational algebra} is an algebra over $\k$ of the form:
$$\k[x]_h=\{f/ h^m\mid m \;\text{is a positive integer}, f\in \k[x]\}, $$
the support of a rational algebra is defined by $$S(\k[x]_h)=\{\lambda\in\k\mid h(\lambda)\neq0\}.$$ For $\lambda\in S(\k[x]_h)$, the simple $\k[x]_h$-module $\k[x]/(x-\lambda)$ will be denoted by $S_{\lambda}$.

Now let us recall the concept of derived representation type; see for instance \cite{Bau07, Vos01}.

We shall say that $A$ is \textit{derived discrete} if for any positive element $\underline{d}$ of $\mathcal{K}_0(A)^{(\Bbb{Z})}$, there are only finitely many isoclasses of indecomposables $X\in\mathcal{D}^b(A)$ such that $\underline{\dim}X=\underline{d}$. $A$ is called \textit{derived tame}, if for any $\underline{d}$ of $\mathcal{K}_0(A)^{(\Bbb{Z})}$, there is a finite set of rational algebras $R_u, u=1,\cdots, s$ and for each $u$ a bounded complex $M_u$ of $A$-$R_u$-bimodules free finitely generated over $R_u$, such that for almost all isomorphism classes $[X]$ with $\underline{\dim}X=\underline{d}$ there is a $\lambda\in S(R_u)$ with $X\cong M_u\otimes_{R_u}S_\lambda$ for some $u\in\{1, \cdots, s\}$. $A$ is called \textit{derived wild} if there is a bounded complex $W$ of $A$-$\k\langle x, y\rangle$-bimodules free finitely generated over $\k\langle x, y\rangle$ such that the functor $$W\otimes_{\k\langle x, y\rangle}-:\k\langle x, y\rangle \text{-mod}\rightarrow \mathcal{D}^b(A)$$ preserves isoclasses and indecomposables.

Note that every derived discrete algebra is derived tame.
\subsection{Monoidal triangulated categories}\label{subsection2.2}
Recall that a \textit{monoidal triangulated category} $\mathcal{C}$ is a triangulated category having a monoidal structure $$\otimes: \mathcal{C}\times \mathcal{C}\rightarrow \mathcal{C}$$ and a unit object $1\in \mathcal{C}$, such that the bifunctor $-\otimes-$ is exact in each variable (see, for example, \cite{NVY22a, NVY22b, Vas24}).
Two monoidal triangulated categories $\mathcal{C}$ and $\mathcal{C}^\prime$ are said to be \textit{monoidal triangulated equivalent} if there is a monoidal functor making $\mathcal{C}$ and $\mathcal{C}^\prime$ be triangulated equivalent.

Let $H$ be a finite-dimensional Hopf algebra. It is well-known that the bounded derived category $\mathcal{D}^b(H)$ is a monoidal triangulated category. See, for example, \cite[Example 3.3]{Big07}. The tensor product in $\mathcal{D}^b(H)$ is defined as follow:
for any $X^\bullet=(X^n, d^n_X)_{n\in\Bbb{Z}}, Y^\bullet=(Y^m, d^m_Y)_{m\in\Bbb{Z}}\in \mathcal{D}^b(H)$,  $$(X^\bullet\otimes Y^\bullet)^n:=\bigoplus_{i+j=n}X^i\otimes Y^{j}$$
with differential
$$d\mid_{X^i\otimes Y^j}=d^i_{X}\otimes \id_{Y^j} +(-1)^i\id_{X^i}\otimes d_Y^j.$$
Let $[1]$ be the shift functor of $\mathcal{D}^b(H)$, i.e. $(X^\bullet[1])^n=X^{n+1}$ for any $n\in \Bbb{Z}$ and $d_{X^\bullet[1]}=-d_{X^\bullet}$. There exists a natural isomorphism $$X^\bullet[1]\otimes Y^\bullet\cong (X^\bullet\otimes Y^\bullet)[1]\cong X^\bullet\otimes Y^\bullet[1]$$ in $\mathcal{D}^b(H)$.
\subsection{Hopf algebras with the (dual) Chevalley property}\label{subsection2.3}
Recall that a finite-dimensional Hopf algebra is said to have the \textit{Chevalley property}, if its radical is a Hopf ideal. A finite-dimensional Hopf algebra is said to have the \textit{dual Chevalley property}, if its coradical is a Hopf subalgebra. According to \cite[Propersition 4.2]{AEG01}, we know that a finite-dimensional Hopf algebra $H$ has the Chevalley property if and only if $H^*$ has the dual Chevalley property.

Let $H$ be a coalgebra over $\k$ and $\mathcal{S}$ be the set of all the simple subcoalgebras of $H$. The \textit{link quiver} $\mathrm{Q}(H)$ of $H$ is defined as follows: the vertices of $\mathrm{Q}(H)$ are the elements of $\mathcal{S}$; for any simple subcoalgebra $C, D\in \mathcal{S}$ with $\dim_{\k}(C)=r^2, \dim_{\k}(D)=s^2$, there are exactly $\frac{1}{rs}\dim_{\k}((C\wedge D)/(C+D))$ arrows from $D$ to $C$ (\cite[Definition 4.1]{CHZ06}).
A subcoalgebra $H^\prime$ of $H$ is called \textit{link-indecomposable} if the link quiver $\mathcal{Q}(H^\prime)$ of $H^\prime$ is connected (as an undirected graph).
A \textit{link-indecomposable component} of $H$ is a maximal link-indecomposable subcoalgebra (\cite[Definition 1.1]{Mon95}).

According to \cite[Theorem 3.2]{Mon95}, the link-indecomposable component $H_{(1)}$ containing $\k1$ must be a normal Hopf subalgebra for any pointed Hopf algebra $H$. By \cite[Proposition 3.16]{Li22}, when $H$ is a Hopf algebra with the dual Chevalley property, $H_{(1)}$ is still a Hopf subalgebra of $H$. However, concerning a non-pointed Hopf algebra $H$ with the dual Chevalley property, $H_{(1)}$ may not necessarily be normal (see for example \cite[Example 6.1]{YLL24}). This tells us that the Hopf algebras with the (dual) Chevalley property is really a nontrivial generalization of elementary (pointed) Hopf algebras in general.

\subsection{Comonomial Hopf algebras}\label{subsection2.4}
Let $\mathrm{Q}=(\mathrm{Q}_0, \mathrm{Q}_1)$ be a finite quiver. Note that we read paths in $\mathrm{Q}$ from right to left. Denote by $\k \mathrm{Q}^a$ and $\k\mathrm{Q}^c$ the path algebra of $\mathrm{Q}$ and the path coalgebra of $\mathrm{Q}$, respectively.

Recall that the counit and comultiplication of path coalgebra $\k\mathrm{Q}^c$ are defined by $\varepsilon(e)=1, \ \Delta(e)=e \otimes e$ for
each $e \in Q_0,$ and for each nontrivial path $p=a_n \cdots a_1, \varepsilon(p)=0,$
\begin{eqnarray*}
\Delta(a_n \cdots a_1)=p \otimes s(a_1) + \sum_{i=1}^{n-1}a_n \cdots
a_{i+1} \otimes a_i \cdots a_1  + t(a_n) \otimes p  .
\end{eqnarray*}

A subcoalgebra $C$ of $\k \mathrm{Q}^c$ is called \textit{comonomial} (\cite[Definition 1.2]{CHYZ04}) provided that the following conditions are satisfied:
\begin{itemize}
  \item[(1)]$C$ contains all vertices and arrows in $\mathrm{Q}$;
  \item[(2)]$C$ is contained in subcoalgebra $C_d(\mathrm{Q}):=\bigoplus\limits_{i=0}^{d-1}\k \mathrm{Q}(i)$ for some $d\geq 2$, where $\mathrm{Q}(i)$ is the set of all paths of length $i$ in $\mathrm{Q}$;
  \item[(3)]$C$ has a basis consisting of paths.
\end{itemize}
A finite-dimensional Hopf algebra is called \textit{comonomial} if it is comonomial as a coalgebra.

The authors of \cite{CHYZ04} constructed non-cosemisimple comonomial Hopf algebras via group data. Let us briefly recall their results.

A \textit{group datum} (see \cite[Definition 5.3]{CHYZ04}) over $\k$ is defined to be a sequence $\alpha=(G, g, \chi, \mu)$ consisiting of
\begin{itemize}
  \item[(1)]a finite group $G$, with an element $g$ in its center;
  \item[(2)]a one-dimensional $\k$-representation $\chi$ of $G$; and
  \item[(3)]an element $\mu\in \k$ such that $\mu=0$ if $o(g)=o(\chi(g))$, and that if $\mu\neq 0$, then $\chi^{o(\chi(g))}=1$.
  \end{itemize}

For a group datum $\alpha=(G, g, \chi, \mu)$ over $\k$, the authors of \cite{CHYZ04} gave the corresponding Hopf algebra structure $A(\alpha)$ as follow.
\begin{definition}\emph{(}\cite[Subsection 5.7]{CHYZ04}\emph{)}
For a group datum $\alpha=(G, g, \chi, \mu)$ over $\k$, define $A(G, g, \chi, \mu)$ to be an associative algebra with generators $x$ and all $h\in G$, with relations
$$
 x^d=\mu(1-g^d),\;\; xh=\chi(h)hx,\;\; \forall h\in G,
$$
where $d=o(\chi(g))$.
The comultiplication $\Delta$, counit $\varepsilon$, and the antipode $S$ are given by
$$
\Delta(g)=g\otimes g,\;\;
\Delta(x)=x\otimes 1+g\otimes x,\;\;\varepsilon(g)=1,\;\; \varepsilon(x)=0,\;\;
S(g)=g^{-1},\;\;S(x)=-xg^{-1}.
$$
\end{definition}
If $\mu(1-g^d)=0$ we say that $A(G, g, \chi, \mu)$ is of \textit{nilpotent type}; otherwise we say that $A(G, g, \chi, \mu)$ is of \textit{non-nilpotent type} (see \cite{KR06, WLZ14}).

\begin{remark}
In the case $A(G, g, \chi, \mu)$ is of nilpotent type, then it is either $\mu=0$ or $1-g^d=0$. In both cases, \cite[Proposition 2.1]{WLZ14} implies that the Hopf algebras constructed from $(G, \chi, g, \mu)$ and $(G, \chi, g, 0)$ respectively are isomorphic.
Moreover, if $A(G, g, \chi, \mu)$ is of non-nilpotent type, it follows from \cite[Lemma 5.8]{CHYZ04} that $A(G, g, \chi, \mu)\cong A(G, g, \chi, 1)$ as Hopf algebras (see also \cite[Corollary 1]{KR06}). Because of this fact, we may assume that $\mu=0$ or $\mu=1$.
\end{remark}
The authors in \cite{KR06} examined the dual of $A(G, g, \chi, \mu)$. Let $\{p_{hx^m}\mid h\in G, 0\leq m< d\}$ denote the basis for $(A(G, g, \chi, \mu))^*$ dual to the basis $\{hx^m\mid h\in G, 0\leq m< d\}$ for $A(G, g, \chi, \mu)$. Define $\xi\in (A(G, g, \chi, \mu))^*$ by $\xi(hx^m)=\delta_{m, 1}$ for all $h\in G$ and $0\leq m< d$. By slight abuse the notation we will regard the one-dimensional $\k$-representation $\chi$ of $G$ as an element of $(A(G, g, \chi, \mu))^*$ by setting $\chi(hx^m)=\chi(h)\delta_{m, 0}$ for all $h\in G$ and $0\leq m< d$.

\begin{lemma}\emph{(}\cite[Proposition 2]{KR06}\emph{)}\label{lem:Hopfstr}
Let $\alpha=(G, g, \chi, \mu)$ be a group datum. As an algebra $(A(\alpha))^*$ may be regard as being generated by symbols $p_h$ for $h\in G$ and $\xi$ subject to the relations
$$
\sum_{h\in G}p_h=1,\;\; p_{h}p_{h^\prime}=\delta_{h, h^\prime}p_h,\; \text{and}\;\;\xi p_h= p_{gh} \xi
$$
for all $h, h^\prime\in G$ and $\xi^n=0$. The coalgebra structure of $(A(\alpha))^*$ is determined by
$$
\Delta(\xi)=\xi\otimes \chi+1\otimes \chi
$$
and
$$
\Delta(p_h)=\sum\limits_{u\in G}p_u\otimes p_{u^{-1}h}+\sum\limits_{u\in G}(\sum\limits_{r+s=d}\frac{\mu \chi(u^{-1}h)^r}{(r)_q!(s)_q!}\xi^rp_u\otimes \xi^s(p_{u^{-1}g^{-d}h}-p_{u^{-1}h}))
$$
for all $h\in G,$ where $$(r)_q!:=1_q2_q\cdots r_q,\;\; (0)_q!:=1,\;\;r_q=1+q+\cdots+q^{r-1}.$$
\end{lemma}

Denote by $\Bbb{Z}_n$ the basic cycle of length $n$, i.e., a quiver with $n$ vertices $e_0, e_1, \cdots, e_{n-1}$ and $n$ arrows $a_0, a_1, \cdots a_{n-1}$, where the arrow $a_i$ goes from the vertex $e_i$ to the vertex $e_{i+1}$. In the following, denote $C_d(\Bbb{Z}_n)$ by $C_d(n)$.
According to \cite[Lemma 5.8]{CHYZ04}, we know that $A(G, g, \chi, \mu)\cong C_d(n)\oplus\cdots \oplus C_d(n)$ as coalgebras, where $n=o(g)$ and $d=o(\chi(g))$.

Let $q\in \k$ be an $n$-th root of unit of order $d$. In \cite{AS98} and \cite{Rad77}, Radford and Andruskiewitsch-Schneider have considered the following Hopf algebra $A(n, d, \mu, q)$ which as an associative algebra is generated by $g$ and $x$ with relations
$$
g^n=1, \;\;\;\;x^d=\mu(1-g^d),\;\;\;\;xg=qgx.
$$
Its comultiplication $\Delta$, counit $\varepsilon$, and the antipode $S$ are given by
$$
\Delta(g)=g\otimes g,\;\; \varepsilon(g)=1,\;\;
\Delta(x)=x\otimes 1+g\otimes x,\;\; \varepsilon(x)=0,\;\;
S(g)=g^{-1},\;\;S(x)=-xg^{-1}.
$$
In fact, $(\Bbb{Z}_n, \overline{1}, \chi, \mu)$ with $\chi(\overline{1})=q$ is a group datum and  $A(n, d, \mu, q)=A(\Bbb{Z}_n, \overline{1}, \chi, \mu)$.

\section{Derived representation type of Hopf algebras with the Chevalley property}\label{section3}
In this section, we try to classify finite-dimensional Hopf algebras with the Chevalley property according to their derived representation type.

Let $H$ be a finite-dimensional Hopf algebra. According to \cite[Theorem 2.1.3]{Mon93}, $H$ is a Frobenius algebra, which is self-injective.
Using \cite[Corollary 2.5]{Bau07}, we have the following lemma.
\begin{lemma}\label{lem:discreteortame}
Suppose $H$ is a finite-dimensional Hopf algebra over $\k$, then either $H$ is derived discrete or derived wild.
\end{lemma}

It is clear that any finite-dimensional semisimple Hopf algebra $H$ is derived discrete. In fact, $\{S[i]\mid S\;\text{is a simple module of }H\text{-mod}\}$ is the set of indecomposable objects in $\mathcal{D}^b(H)$. Next we will focus on the general cases.

Recall that a finite-dimensional algebra $A$ is said to be of \textit{finite representation type} if there are finitely many isoclasses of indecomposable $A$-modules. An algebra $A$ is said to be of \textit{infinite representation type}, if $A$ is not of finite representation type. A finite-dimensional coalgebra $C$ is said to be of \textit{finite corepresentation type}, if the dual algebra $C^*$ is of finite representation type. A finite-dimensional coalgebra $C$ is defined to be of \textit{infinite corepresentation type},
if $C^*$ is of infinite representation type.

It is direct to see the following lemma.
\begin{lemma}\label{lem:finiterep}
Suppose $H$ is a finite-dimensional derived discrete Hopf algebra over $\k$, then $H$ is of finite representation type.
\end{lemma}
\begin{proof}
Assume the contrary, i.e., $H$ is of infinite representation type. By \cite[Theorem 2.4]{Bau85}, there is an infinitely family of isomorphism classes of finite-dimensional indecomposable left modules over $H$ with a dimension vector $\underline{v}$ in the Grothendieck group $\mathcal{K}_0(H)$. Note that there is an inclusion from $H$-mod to $\mathcal{D}^b(H)$. This means that $H$ is not derived discrete, which is a contradiction. Thus $H$ is of finite representation type.
\end{proof}
With the help of the preceding lemma, we can now prove:
\begin{lemma}\label{thm:discrete}
Suppose $H$ is a finite-dimensional non-semisimple Hopf algebra over $\k$ with the Chevalley property. If $H$ is derived discrete, then the link-indecomposable component $(H^*)_{(1)}$ containing $\k1$ of $H^*$ is isomorphic to $A(n, 2, \mu, -1)$ as Hopf algebras.
\end{lemma}
\begin{proof}
By Lemma \ref{lem:finiterep}, we know that $H$ is of finite representation type. Note that the category of finite-dimensional left $H$-modules is isomorphic to the category of finite-dimensional right $H^*$-comodules. This means that $H^*$ is a a finite-dimensional non-cosemisimple Hopf algebra over $\k$ with the dual Chevalley property of finite corepresentation type. It follows from \cite[Theorem 5.15]{YLL24} that $$(H^*)_{(1)}\cong A(n, d, \mu, q)$$ as Hopf algebras. Let $\mathcal{S}$ be the set of all the simple subcoalgebras of $H^*$. According to \cite[Corollary 4.10]{Li22}, there exists a subset $\mathcal{S}_0$ containing $\k1$ of $\mathcal{S}$ such that
$$H^*=(H^*)_{(1)}\oplus(\bigoplus\limits_{C\in\mathcal{S}_0\backslash \{\k1\}}C(H^*)_{(1)}).$$
This means that $$H\cong (A(n, d, \mu, q))^*\oplus(\bigoplus\limits_{C\in\mathcal{S}_0\backslash \{\k1\}}(CA(n, d, \mu, q))^*).$$
It follows that $(A(n, d, \mu, q))^*$ is derived discrete. Note that $A(n, d, \mu, q)\cong C_d(n)$ as coalgebras. Using \cite[Lemma 1.3]{CHYZ04}, we obtain $$(A(n, d, \mu, q))^*\cong \k \Bbb{Z}_n^a/ J^d$$ as algebras, where $J$ is the ideal generated by arrows.
From \cite[Lemma 3.2]{BGV21}, $\k\Bbb{Z}_n^a/J^d$ is derived tame if and only if $d=2$. Since Lemma \ref{lem:discreteortame} implies that $(A(n, d, \mu, q))^*$ is either derived discrete or derived wild, it follows that $(A(n, d, \mu, q))^*$ is of derived discrete if and only if $d=2$ and $q=-1$.
As a conclusion, we have $$(H^*)_{(1)}\cong A(n, 2, \mu, -1)$$as Hopf algebras.
\end{proof}
We say that a finite-dimensional $\k$-algebra $A$ is \textit{indecomposable} (or \textit{connected}) if $A$ is not a direct product of two algebras. We are in a position to show our main conclusion now:
\begin{proof}[Proof of Theorem \ref{thm:indecom}]
``If part": An argument similar to the one used in the proof of Lemma \ref{thm:discrete} shows that $(A(n, 2, \mu, -1))^*$ is derived discrete.

``Only if part": Let $\mathcal{S}$ be the set of all the simple subcoalgebras of $H^*$. According to \cite[Corollary 4.10]{Li22}, there exists a subset $\mathcal{S}_0$ containing $\k1$ of $\mathcal{S}$ such that
$$H^*=(H^*)_{(1)}\oplus(\bigoplus\limits_{C\in\mathcal{S}_0\backslash \{\k1\}}C(H^*)_{(1)}).$$ Since $H$ is indecomposable, it follows that $$H\cong ((H^*)_{(1)})^*.$$ Using Lemma \ref{thm:discrete}, we can show that $$H\cong (A(n, 2, \mu, -1))^*$$ as Hopf algebras.
\end{proof}
Now let us focus on a special situation.
\begin{corollary}\label{coro:elementary}
Let $H$ be a finite-dimensional non-semisimple elementary Hopf algebra over $\k$. Then $H$ is derived discrete if and only if $H\cong (A(\alpha))^*$ for some group datum $\alpha=(G, g, \chi, \mu)$ with $o(\chi(g))=2$.
\end{corollary}
\begin{proof}
Suppose $H\cong (A(\alpha))^*$ for some group datum $\alpha=(G, g, \chi, \mu)$ with $o(\chi(g))=2$. According to \cite[Lemma 5.8]{CHYZ04}, we know that $$A(\alpha)\cong C_2(n)\oplus\cdots \oplus C_2(n)$$ as coalgebras, where $n=o(g)$. From \cite[Lemma 1.3]{CHYZ04}, we obtain $$(A(\alpha))^*\cong \k \Bbb{Z}_n^a/ J^2\oplus\cdots\oplus \k \Bbb{Z}_n^a/ J^2$$ as algebras, where $J$ is the ideal generated by arrows in $\Bbb{Z}_n$.
Using the same argument as in the proof of Lemma \ref{thm:discrete}, we know that $H$ is derived discrete. Conversely, since $H$ is derived discrete, it follows from Lemma \ref{lem:finiterep} that $H$ is of finite representation type. This means that $H^*$ is of finite corepresentation type. By \cite[Theorem 4.6]{LL07}, we have $H^*\cong A(\alpha)$ for some group datum $\alpha=(G, g, \chi, \mu)$. From \cite[Lemma 5.8]{CHYZ04}, we know that $$A(\alpha)\cong C_d(n)\oplus\cdots \oplus C_d(n)$$ as coalgebras, where $n=o(g)$ and $d=o(\chi(g))$. A same argument as in the proof of Lemma \ref{thm:discrete} shows that $o(\chi(g))=2$.
\end{proof}
It is not difficult to verify the following proposition.
\begin{proposition}\label{prop:equiv}
Let $\alpha=(G, g, \chi, \mu), \alpha^\prime=(G^\prime, g^\prime, \chi^\prime, \mu^\prime)$ be two group datums with $o(\chi(g))=2$ and $o(\chi^\prime(g^\prime))=2$, respectively. Then
 $\mathcal{D}^b((A(\alpha))^*)$ is equivalent to $\mathcal{D}^b((A(\alpha^\prime))^*)$ as triangulated categories if and only if $A(\alpha)\cong A(\alpha^\prime)$ as coalgebras. In particular, $\mathcal{D}^b((A(n, 2, \mu, -1))^*)$ is triangulated equivalent to $\mathcal{D}^b((A(m, 2, \mu^\prime, -1))^*)$ if and only if $n=m$.
\end{proposition}
\begin{proof}
Using the same argument as in the proof of Corollary \ref{coro:elementary}, we know that $$(A(\alpha))^*\cong \k \Bbb{Z}_n^a/ J^2\oplus\cdots\oplus \k \Bbb{Z}_n^a/ J^2$$ as algebras, where $J$ is the ideal generated by arrows in $\Bbb{Z}_n$ and $n=o(g)$. According to \cite[Theorem A]{BGS04}, $\mathcal{D}^b(\k\Bbb{Z}_m^a/ J^2)$ is triangulated equivalent to $\mathcal{D}^b(\k\Bbb{Z}_n^a/ J^2)$ if and only if $m=n.$ It turns out that $\mathcal{D}^b((A(\alpha))^*)$ is equivalent to $\mathcal{D}^b((A(\alpha^\prime))^*)$ as triangulated categories if and only if $A(\alpha)\cong A(\alpha^\prime)$ as coalgebras.
\end{proof}
\begin{remark}\label{rm:monoidalequ}
According to \cite{Wan18}, we know that $(A(G, g, \chi, 0))^*$-mod is tensor equivalent to $(A(G, g, \chi, 1))^*$-mod. This means that $\mathcal{D}^b((A(G, g, \chi, 0))^*)$ is monoidal triangulated equivalent to $\mathcal{D}^b((A(G, g, \chi, 1))^*)$.
With the notations in Proposition \ref{prop:equiv}, we know that $\mathcal{D}^b((A(\alpha))^*)$ is monoidal triangulated equivalent to $\mathcal{D}^b((A(\alpha^\prime))^*)$ if and only if $A(\alpha)\cong A(\alpha^\prime)$ as coalgebras, if and only if $(A(\alpha))^*$-mod is tensor equivalent to $(A(\alpha^\prime))^*$-mod.
\end{remark}
Recall that a \textit{gentle quiver} $(\mathrm{Q}, I)$ (\cite{AS87}) is formed by a quiver $\mathrm{Q}$ and an ideal $I$ of $\k \mathrm{Q}^a$ with properties $(a)$-$(d)$:
\begin{itemize}
  \item[(a)]At every vertex of $\mathrm{Q}$ at most two arrows stop and at most two arrows start.
  \item[(b)]$I$ is generated by paths of length two.
  \item[(c)]For every arrow $\beta$, there is at most one arrow $\alpha$ with $\beta\alpha$ not in $I$ and at most one arrow $\gamma$ with $\gamma\beta$ not in $I$.
  \item[(d)]For every arrow $\beta$, there is at most one arrow $\alpha^\prime$ with $\beta\alpha^\prime$ in $I$ and at most one arrow $\gamma^\prime$ with $\gamma^\prime\beta$ in $I$.
\end{itemize}
A $\k$-algebra $A$ is called \textit{gentle} if $A\cong \k \mathrm{Q}^a/ I$, where the pair $(\mathrm{Q}, I)$ is gentle.

As an application of Corollary \ref{coro:elementary}:
\begin{corollary}\label{coro:gentle}
Let $H=\k \mathrm{Q}^a/ I$ be a finite-dimensional gentle algebra over $\k$. Then algebra $H$ admits a Hopf algebra structure if and only if $\mathrm{Q}$ is a disjoint union of basic cycles of length $n$ and $I=J^2$, where $n$ is an even number and $J$ is the ideal generated by arrows.
\end{corollary}
\begin{proof}
``Only if part": According to \cite[Theorem 3]{BM03}, we know that $H$ is of derived tame type. Since $H=\k \mathrm{Q}^a/ I$ is a Hopf algebra, it follows from Lemma \ref{lem:discreteortame} that $H$ is either derived discrete or derived wild. This implies that $H$ is derived discrete. Using Corollary \ref{coro:elementary}, we know that $$H\cong (A(\alpha))^*$$ as Hopf algebras, where $\alpha=(G, g, \chi, \mu)$ is a group datum with $o(\chi(g))=2$. Note that $A(\alpha)\cong C_n(2)\oplus\cdots\oplus C_n(2)$ as coalgebras. It follows from \cite[Lemma 1.3]{CHYZ04} that $\mathrm{Q}$ is a disjoint union of basic cycles of length $n$ and $I=J^2$, where $n$ is an even number.

``If part": Suppose $\mathrm{Q}$ is a disjoint union of basic cycles of length $n$ and $I=J^2$, where $n$ is an even number. According to \cite[Theorem 5.1]{CHYZ04}, $C_n(2)\oplus\cdots\oplus C_n(2)$ admits a Hopf algebra structure. It is clear that $(C_n(2)\oplus\cdots\oplus C_n(2))^*$ also admits a Hopf algebra structure.
\end{proof}

\section{Characterization of $\mathcal{D}^b((A(n, 2, \mu, -1))^*)$}\label{section4}
In this section, let $H= (A(n, 2, \mu, -1))^*$. We will give a description for the indecomposable objects in the bounded derived category of $H$ and determine their tensor products.

Note that $H\cong \k \Bbb{Z}_n^a/J^2$ as algebras, where $J$ is the ideal generated by arrows. This means that we can describe $H$ by quiver and relations: The quiver is cyclic, with $n$ vertices $e_0, e_1, \cdots, e_{n-1}$ and $n$ arrows $a_0, a_1, \cdots, a_{n-1}$, where the arrow $a_i$ goes from the vertex $e_i$ to the vertex $e_{i+1}$, and we factor by the ideal generated by all paths of length $2$. The indices are viewed as elements in the cyclic group $\Bbb{Z}_n$ and are written mod $n$. Notice that if $\{p_{g^ix^j}\mid 0\leq i\leq n-1, 0\leq j\leq 1\}$ denotes the basis for $(A(n, 2, \mu, -1))^*$ dual to the basis $\{g^ix^j\mid 0\leq i\leq n-1, 0\leq j\leq 1\}$ for $A(n, 2, \mu, -1)$, the correspondence is determined by $p_{g^i}\mapsto e_i$ and $p_{g^ix}\mapsto a_i$.

Let $P_i$ be the indecomposable projective left $H$-module corresponding to $e_i$. That is, $P_i=He_i=\span\{e_i, a_{i}\}$.
As is well known, the paths $\omega\in\k\Bbb{Z}_n^a/ J^2$, $s(\omega)=e_i, t(\omega)=e_j$ form a basis of the morphism space $\operatorname{Hom}_H(P_j, P_i)$. See, for instance, \cite[Section 3]{BM03}. Namely, for any path $\omega\in\k\Bbb{Z}_n^a/ J^2$, it defines the morphism
$r_\omega:P_j\rightarrow P_i, u\mapsto u\omega $. From now on, by abuse of the notation, we shall identify a path $\omega$ with its corresponding map $r_\omega$. It is clear that $\operatorname{Im}(a_{i+1})=\ker(a_{i})=\span\{a_{i+1}\}.$ Moreover, we have left $H$-module isomorphisms $$\ker(a_{i-2})\cong P_i/ \operatorname{Im}(a_i)\cong S_i,$$ where $S_i$ is the simple left $H$-module corresponding to $e_i$.

It is apparent that $\k\Bbb{Z}_n^a/J^2$ is a gentle algebra.
In \cite[Theorem 3]{BM03}, Bekkert and Merklen described the indecomposable objects in the bounded derived category of a finite-dimensional gentle algebra. Their crucial observation is that it is enough to consider complexes where the differential is given by matrices whose entries are either zero or a path. The indecomposable objects are completely classified in terms of homotopy strings and bands. See \cite{Bob11} for the terminology and \cite[Section 2]{ALP16} for an overview.
The homotopy strings for $\mathcal{D}^b(\k\Bbb{Z}_n^a/J^2)$ are listed in \cite[Lemma 7.1]{ALP16}.

Combining \cite[Theorem 3]{BM03} and \cite[Lemma 7.1]{ALP16}, we have:
\begin{proposition}\label{prop:indecom}
The indecomposable objects in $D^b((A(n, 2, \mu, -1))^*)$ up to isomorphism are of the following form:
$$
M_{i, i-j}[p],
N_{i}[p],
$$
where \begin{eqnarray*}
M_{i, i-j}&=&0\rightarrow P_i\xrightarrow{a_{i-1}}P_{i-1}\xrightarrow{a_{i-2}} P_{i-2}\cdots  \xrightarrow{a_{i-j}}P_{i-j}\rightarrow 0,\\
N_i&=&0\rightarrow \ker(a_{i})\xrightarrow{\eta_{i+1}} P_{i+1} \xrightarrow{a_{i}}P_i\rightarrow 0,
\end{eqnarray*}
for any $i, p\in \Bbb{Z}, j\in \Bbb{Z}_+$, and the leftmost non-zero terms of $M_{i, i-j}$ and $N_{i}$ lie in degree $0$.
\end{proposition}

Now we turn to determine the tensor product of indecomposables in $\mathcal{D}^b(H)$. Without loss of generality, we assume $A(n, 2, \mu, -1))$ is of nilpotent type, that is, $\mu(1-g^2)=0$.
The formulate
\begin{eqnarray*}
\Delta(e_i)&=&\sum_{j+l=i} e_j\otimes e_l,\;\;\varepsilon(e_i)=\delta_{i, 0}, \;\;S(e_i)=e_{-i},\\
\Delta(a_i)&=&\sum_{j+l=i}(e_j\otimes a_l+(-1)^l a_j\otimes e_l),\;\;\varepsilon(a_i)=0,\;\;S(a_i)=(-1)^ia_{-i-1},
\end{eqnarray*}
determine the Hopf algebra structure of $H$ (see \cite{Cil93, EGST06}).
Moreover, we have the following left $H$-module isomorphisms (see \cite[Lemma 3.1 and Proposition 3.3]{Wan18}):
\begin{eqnarray*}
&&S_i\otimes S_j\cong S_{i+j},\\
&&P_i\otimes S_j\cong S_j\otimes P_i \cong P_{i+j},\\
&&P_i\otimes P_j\cong P_{i+j+1}\oplus P_{i+j},
\end{eqnarray*}
where $0\leq i, j\leq n-1$.

Based on the above arguments, we have the following proposition.
\begin{proposition}\label{prop:tensor}
The followings hold in ${D}^b((A(n, 2, \mu, -1))^*)$:
\begin{itemize}
  \item[(1)]$N_i[p]\otimes N_j[q]\cong N_{i+j}[p+q-2]$;
  \item[(2)]$M_{i, i-k}[p]\otimes N_j[q]\cong N_j[q]\otimes M_{i, i-k}[p]\cong M_{i+j, i-k+j}[p+q-2]$;
  \item[(3)]$M_{i, i-k}[p]\otimes M_{j, j-t}[q]\cong M_{i+j+1, i+j+1-\operatorname{min}\{k,t\}}[p+q]\oplus M_{i+j-\operatorname{max}\{k,t\}, i+j-k-t}[p+q-\operatorname{max}\{k,t\}]$,
\end{itemize}
where $i, j, p, q\in \Bbb{Z}$ and $k, t\in \Bbb{Z}_+$.
\end{proposition}
\begin{proof}
\begin{itemize}
  \item[(1)]
According to \cite[Example 3.3 and Theorem 4.1]{Big07}, we know that
 $$ H^m(N_i\otimes N_j)\cong\left\{
\begin{aligned}
S_{i+j},&~~~ \text{if} ~~~ m=4; \\
0\;\;,& ~~~otherwise.
\end{aligned}
\right.
$$
It follows that
$$N_i\otimes N_j\cong N_{i+j}[-2],$$
which means that
$$N_i[p]\otimes N_j[q]\cong (N_{i}\otimes N_j)[p+q]\cong N_{i+j}[p+q-2].$$
\item[(2)]We shall adopt the same procedure as in the proof of $(1)$. Using \cite[Example 3.3 and Theorem 4.1]{Big07}, we have
$$ H^m(M_{i, i-k}\otimes N_j)\cong\left\{
\begin{aligned}
 P_{i+j}\;,& ~~~ \text{if} ~~~k=0~~~ \text{and}~~~ m=2; \\
 S_{i+j+1},&~~~ \text{if} ~~~k\geq1~~~ \text{and}~~~ m=2; \\
 S_{i+j-k},&~~~ \text{if} ~~~k\geq1~~~ \text{and}~~~ m=k+2; \\
0\;\;\;,& ~~~otherwise.
\end{aligned}
\right.
    $$
This indicates that $$M_{i,i-k}\otimes N_j\cong M_{i+j, i-k+j}[-2].$$
A similar argument shows that $$N_j\otimes M_{i,i-k} \cong M_{i+j, i-k+j}[-2].$$
Thus we have $$M_{i, i-k}[p]\otimes N_j[q]\cong N_j[q]\otimes M_{i, i-k}[p]\cong M_{i+j, i-k+j}[p+q-2].$$
\item[(3)]
An argument similar to the one used in the proof of $(1)$ shows that
$$
 H^m(M_{i, i-k}\otimes M_{j, j-t})\cong\left\{
\begin{aligned}
 P_{i+j+1}\oplus P_{i+j},\;\;\;\;\;\;& ~~~ \text{if} ~~~k=t=0~~~ \text{and}~~~ m=0; \\
P_{i+j+1},\;\;\;\;\;\;\;\;\;\;&~~~ \text{if} ~~~k=0, t\geq 1~~~ \text{and}~~~ m=0; \\
P_{i+j-t},\;\;\;\;\;\;\;\;\;\;&~~~ \text{if} ~~~k=0, t\geq 1~~~ \text{and}~~~ m=t; \\
 P_{i+j+1},\;\;\;\;\;\;\;\;\;\;&~~~ \text{if} ~~~k\geq 1, t= 0~~~ \text{and}~~~ m=0; \\
 P_{i+j-k},\;\;\;\;\;\;\;\;\;\;&~~~ \text{if} ~~~k\geq 1, t= 0~~~ \text{and}~~~ m=k; \\
 S_{i+j-t+1}\oplus S_{i+j-k+1},&~~~ \text{if} ~~~k= t\geq 1~~~ \text{and}~~~ m=k; \\
 S_{i+j-k+1},\;\;\;\;\;\;\;\;&~~~ \text{if} ~~~k\neq t\geq 1, t\geq1~~~ \text{and}~~~ m=k; \\
 S_{i+j-t+1},\;\;\;\;\;\;\;\;&~~~ \text{if} ~~~k\neq t\geq 1, t\geq1~~~ \text{and}~~~ m=t; \\
  S_{i+j+2},\;\;\;\;\;\;\;\;\;\;&~~~ \text{if} ~~~k\geq 1, t\geq 1~~~ \text{and}~~~ m=0; \\
  S_{i+j-k-t},\;\;\;\;\;\;\;\;&~~~ \text{if} ~~~k\geq 1, t\geq 1~~~ \text{and}~~~ m=k+t; \\
0,\;\;\;\;\;\;\;\;\;\;\;\;\;\;& ~~~otherwise.
\end{aligned}
\right.
$$
When $k=0$, or $t=0$, or $k=t$, it is straightforward to show that
$$M_{i, i-k}\otimes M_{j, j-t} \cong M_{i+j+1, i+j+1-\operatorname{min}\{k,t\}}\oplus M_{i+j-\operatorname{max}\{k,t\}, i+j-k-t}[-\operatorname{max}\{k,t\}].$$
It remains to consider the case that $k\neq t\geq 1, t\geq1$. Without loss of generality we may assume $k> t\geq1$.
By definition, the complex
 $M_{i, i-k}\otimes M_{j, j-t}$
 equals
 \begin{eqnarray*}
&&0\rightarrow P_i\otimes P_j\xrightarrow{\tiny\left(\begin{array}{cc}
a_{i-1}\otimes 1\\
1\otimes a_{j-1}
 \end{array}
\right)}
P_{i-1}\otimes P_j\oplus P_i\otimes P_{j-1}\\
&&\rightarrow \cdots\\
&&\rightarrow P_{i-t}\otimes P_j\oplus P_{i-t+1}\otimes P_{j-1}\oplus\cdots\oplus P_i\otimes P_{j-t}\\
&&\xrightarrow{L_t}P_{i-t-1}\otimes P_j\oplus P_{i-t}\otimes P_{j-1}\oplus \cdots \oplus P_{i-1}\otimes P_{j-t}\\
&&\rightarrow \cdots\\
&&\xrightarrow{\tiny\left(\begin{array}{cc}
(-1)^{k}1\otimes a_{j-t}& a_{i-k}\otimes 1
 \end{array}
\right)}
P_{i-k}\otimes P_{j-t}\\
&&\rightarrow 0,
\end{eqnarray*}
where $$L_t=\small\left(\begin{array}{ccccc}
a_{i-t-1}\otimes 1\\
(-1)^t1\otimes a_{j-1}&a_{i-t}\otimes 1\\
&(-1)^{t-1}1\otimes a_{j-2}\\
&&\ddots\\
&&&a_{i-1}\otimes 1
 \end{array}
\right).$$
We define $f_{i,j}\in \operatorname{Hom}_{\k}(P_i\otimes P_j, P_{i+j+1}\oplus P_{i+j})$ on the basis by
  \begin{eqnarray*}
  &&f_{i,j}(e_i\otimes e_j)=e_{i+j},\\
  &&f_{i,j}(\frac{1}{2}(-1)^{j-1}e_{i}\otimes a_j+\frac{1}{2}a_i\otimes e_j)=e_{i+j+1},\\
  &&f_{i,j}(e_i\otimes a_j+(-1)^ja_i\otimes e_j)=a_{i+j},\\
  &&f_{i,j}(a_i\otimes a_j)=a_{i+j+1}.
  \end{eqnarray*}
  It is straightforward to show that $f_{i,j}$ is a left $(A(n, 2, 0, -1))^*$-module isomorphism.
  This indicates that the following diagram
  \begin{displaymath}\small
\xymatrix{
{P_{i-t}\otimes P_j\oplus \cdots\oplus P_i\otimes P_{j-t}}\ar[d]^{L_t} \ar[rr]^{\cong}&&{P_{i+j-t+1}\oplus P_{i+j-t}\oplus\cdots\oplus P_{i+j-t}}\ar[d]^{L^\prime_t}\\
{P_{i-t-1}\otimes P_j\oplus \cdots \oplus P_{i-1}\otimes P_{j-t}}\ar[rr]^{\cong}&&{P_{i+j-t}\oplus P_{i+j-t-1}\oplus\cdots\oplus  P_{i+j-t-1}}
}
\end{displaymath}
  commutes, where
  $$L^\prime_t=\small\left(\begin{array}{ccccccccc}
  \frac{(-1)^{j-1}}{2}a_{i+j-t}&1\\
  0&\frac{(-1)^{j}}{2}a_{i+j-t-1}\\
  \frac{(-1)^{t}}{2}a_{i+j-t}&(-1)^{j+t}\\
  0&\frac{(-1)^{t}}{2}a_{i+j-t-1}\\
  &&\ddots\\
 &&& \frac{(-1)^{j-t-1}}{2}a_{i+j-t}&1\\
  &&&0&\frac{(-1)^{j-t}}{2}a_{i+j-t-1}
 \end{array}
\right).$$
For any $r, s\geq1$, $0\leq m\leq n-1$ and $l\in\k$, let $(R_r+la_mR_s)$ be the elementary matrix corresponding to the elementary row operator which adds $la_m$ times row $s$ to row $r$.
Let
\begin{eqnarray*}
K_1&=&(R_{2t+2}+a_{i+j-t-1}R_{2t-1})(R_{2t+1}+(-1)^{j-t+1}R_{2t-1})\cdots\\
&&(R_4+(-1)^{t-1}a_{i+j-t-1}R_1)(R_3+(-1)^{j+t-1}R_1)\\
&&(R_{2t+2}+\frac{(-1)^{j-t+1}}{2}a_{i+j-t-1}R_{2t+1})\cdots(R_{2}+\frac{(-1)^{j-1}}{2}a_{i+j-t-1}R_1),\\
K_2&=&(R_{2t+2}-a_{i+j-t}R_{2t-1})(R_{2t+1}+(-1)^{j-t+1}R_{2t-1})\cdots\\
&&(R_4+(-1)^{t}a_{i+j-t}R_1)(R_3+(-1)^{j+t-1}R_1)\\
&&(R_{2t+2}+\frac{(-1)^{j-t+1}}{2}a_{i+j-t}R_{2t+1})\cdots(R_{2}+\frac{(-1)^{j-1}}{2}a_{i+j-t}R_1).
\end{eqnarray*}
It is not hard to see that
\begin{eqnarray*}
K_1L^\prime_t=\small\left(\begin{array}{ccccccccc}
  0&1\\
  0&0\\
  &&0&1\\
  &&0&0\\
  &&&&\ddots\\
 &&&&& 0&1\\
  &&&&& 0&0
 \end{array}
\right)
K_2.
\end{eqnarray*}
It turns out that the complex
 $M_{i, i-k}\otimes M_{j, j-t}$
  does not contain  $M_{i+j+1, i+j+1-\operatorname{max}\{k,t\}}$ or $M_{i+j+1, i+j-k-t}$ as a direct summand, which follows that
  $$M_{i, i-k}\otimes M_{j, j-t} \cong M_{i+j+1, i+j+1-\operatorname{min}\{k,t\}}\oplus M_{i+j-\operatorname{max}\{k,t\}, i+j-k-t}[-\operatorname{max}\{k,t\}].$$
\end{itemize}
\end{proof}

\end{document}